\documentclass[a4paper,11pt]{article}
\usepackage{amsthm,amsfonts,amssymb,amsmath}

\usepackage{bookmark}
\usepackage{comment}
\usepackage{calligra}%mathcal
\usepackage{chngcntr}
\usepackage{enumitem}
\usepackage[a4paper]{geometry}
\usepackage{indentfirst}%首行空格
\usepackage{mathtools}
\usepackage{mathrsfs}%mathscr
\usepackage{lmodern}
\usepackage{microtype}
\usepackage{stmaryrd}%mapsfrom
\usepackage{tikz}%picture
\usepackage{tikz-3dplot}
\usepackage{tikz-cd}%commutative-diagram
\usepackage[nottoc]{tocbibind}%参考文献section
\usepackage{verbatim}%comments
\usepackage[T1]{fontenc}

\usepackage{hyperref}

\usepackage[capitalise]{cleveref}
\hypersetup{
	linktoc = page,
	colorlinks = true,
	linkcolor = cyan, %Colour of internal links
	urlcolor = [rgb]{1, 0.5, 0.5},
	citecolor = magenta %Colour of citations
}

\theoremstyle{plain}
\newtheorem{theorem}{Theorem}[section]

\newtheorem{lemma}[theorem]{Lemma}
\newtheorem{proposition}[theorem]{Proposition}
\newtheorem{conjecture}[theorem]{Conjecture}

\theoremstyle{definition}
\newtheorem{definition}[theorem]{Definition}
\newtheorem{remark}[theorem]{Remark}

\makeatletter
    \let\c@equation\c@theorem
\makeatother
\numberwithin{equation}{section}
\newcommand{\bb}[1]{\mathbb{#1}}
\newcommand{\rr}[1]{\mathrm{#1}}
\newcommand{\cc}[1]{\mathcal{#1}}

\newcommand{\quash}[1]{}

\AtEndDocument{%
  \par
  \medskip
  \begin{tabular}{@{}l@{}}%
    \textsc{Pengfei Huang}\\
    \textsc{School of Mathematics, Nanjing University},\\
	\textsc{Nanjing 210093, China}\\
    \textit{E-mail address}: \texttt{pfhwangmath@gmail.com}\vspace{4pt} \\

    \textsc{Yichen Qin}\\
    \textsc{School of Mathematical Sciences, Fudan University} \\
    \textsc{Handan Road 220, 200437 Shanghai China}\\
    \textit{E-mail address}: \texttt{yichenqin@fudan.edu.cn} \vspace{4pt} \\

    \textsc{Hao Sun}\\
    \textsc{Department of Mathematics, South China University of Technology}\\
    \textsc{Guangzhou, 510641, China} \\
    \textit{E-mail address}: \texttt{hsun71275@scut.edu.cn} \\
  \end{tabular}}

\title{Rigid $G$-connections and nilpotency of $p$-curvatures}
\author{Pengfei Huang, Yichen Qin, Hao Sun
}

\date{ }

\begin{document}
\maketitle
\begin{abstract}
Motivated by Simpson's conjecture on the motivicity of rigid irreducible connections, Esnault and Groechenig demonstrated that the mod-$p$ reductions of such connections on smooth projective varieties have nilpotent $p$-curvatures. In this paper, we extend their result to integrable $G$-connections on curves.
\end{abstract}
\tableofcontents

\renewcommand{\thefootnote}{\fnsymbol{footnote}}

\section{Introduction}

Let $X$ be a smooth complex projective variety. The celebrated Riemann--Hilbert correspondence gives rise to an analytic isomorphism $\mathcal{M}_{\mathrm{dR}}(X,n) \cong \mathcal{M}_{\mathrm{B}}(X,n)$, where $\mathcal{M}_{\mathrm{dR}}(X,n)$ is the moduli space of integrable connections of rank $n$ on $X$ and $\mathcal{M}_{\mathrm{B}}(X,n)$ is the moduli space of complex local systems of rank $n$ on $X$ \cite{Sim94b}. When $X$ has higher dimension, the moduli spaces $\mathcal{M}_{\mathrm{dR}}(X,n)$ and $\mathcal{M}_{\mathrm{B}}(X,n)$ may have isolated points, whose associated integrable connections (or local systems)  are called \textit{rigid}. Following the nonabelian Hodge correspondence, any rigid integrable connection underlies a complex variation of Hodge structure. Furthermore, it is a complex direct factor of a $\mathbb{Q}$-variation of Hodge structure \cite[Theorem 5]{Sim92}. This leads to the notion of motivicity from a geometric aspect.

More precisely, an integrable connection $(E,\nabla)$ is \textit{motivic} (or \textit{of geometric origin}) if there exists a dense open subvariety $U\subset X$ and a smooth projective morphism $f: Z\to U$ such that $(E,\nabla)$ is a direct summand of $R^if_*\mathcal{O}_Z$ equipped with the Gauss--Manin connection.

Returning to the Riemann--Hilbert correspondence, notice that taking monodromy representations for integrable connections involves exponentiation, which is transcendental. Simpson addressed this transcendental nature in \cite{Sim90}, where he posed a question about which integrable connections defined over $\overline{\mathbb{Q}}$ have associated monodromy representations also defined over $\overline{\mathbb{Q}}$. He conjectured that such integrable connections should originate from geometry, namely the following conjecture:

\begin{conjecture}[Simpson's standard conjecture, \cite{Sim90}]\label{Sim_standard}
    By spreading out,  the $\overline{\mathbb{Q}}$-points in the intersection $\mathcal{M}_{\mathrm{B}}(X,n)(\overline{\mathbb{Q}})\bigcap \mathcal{M}_{\mathrm{dR}}(X,n)(\overline{\mathbb{Q}})$ are motivic.
\end{conjecture}

As we have seen, each rigid integrable connection corresponds to an irreducible component of $\mathcal{M}_{\mathrm{dR}}(X,n)$. Therefore, if $X$ is defined over $\overline{\mathbb{Q}}$, each of these zero-dimensional components is also defined over $\overline{\mathbb{Q}}$. This leads to a weaker version of the standard conjecture, stated as follows:

\begin{conjecture}[Simpson's motivicity conjecture, \cite{Sim92}]\label{Sim_motivic}
    Rigid integrable connections are motivic.
\end{conjecture}

The motivicity conjecture \ref{Sim_motivic} is known for local systems on $\mathbb{P}^1$ with finitely many punctures by Katz \cite{Kat96}, rank $2$ and rank $3$ local systems on quasi-projective varieties from the work of Corlette--Simpson \cite{CS08} and Langer--Simpson \cite{LS18} respectively. We also refer the reader to \cite{Esn23a,Esn23b,Sim90} for more details about the background and recent progress.

To provide general evidence for the motivicity conjecture, it is useful to examine whether rigid connections exhibit similar properties to those of motivic connections. Here, the work of Esnault and Groechenig stands out as deeply inspirational. For example, they proved that cohomologically rigid connections are integral \cite{EG18}. Moreover, since Gauss--Manin connections in characteristic $p$ have nilpotent $p$-curvatures \cite{Katz72}, Esnault and Groechenig also confirmed the nilpotency of $p$-curvatures of mod-$p$ reductions of rigid connections \cite{EG20}.

Progress has also been made on analogous problems for $G$-connections. Specifically, Klevdal and Patrikis proved the integrality of cohomologically rigid $G$-connections \cite{CP22}, and F{\ae}rgeman established the motivicity conjecture for $G$-connections on curves by demonstrating that rigid $G$-connections on curves are Hecke eigensheaves \cite{Fae24}.

\hfill

In this paper, we aim to generalize Esnault--Groechenig's result on nilpotency of $p$-curvatures \cite[Theorem 1.4]{EG20} to rigid integrable $G$-connections on curves. Although this result can be implied by results in \cite{Fae24}, we provide a new proof and give some evidence that this new strategy can be generalized to higher dimensional case in \cref{sect_higher_dim}. Here is the main result:

\begin{theorem}\label{thm:main}
  Let $X$ be a connected complex smooth projective curve, $G$ a connected complex reductive group, and $(E,\nabla)$ a rigid integrable $G$-connection. Then there is a scheme $S$ of finite type over $\bb{Z}$ over which $(X,(E,\nabla))$ has a model $(X_S, (E_S,\nabla_S))$ such that for all closed points $s\in |S|$, the restrictions $(E_s,\nabla_s)$ have nilpotent $p$-curvatures.
\end{theorem}

The proof of Theorem \ref{thm:main} follows a similar strategy to that used for \cite[Theorem 1.4]{EG20}. Below is a summary of this approach for integrable $G$-connections.
\begin{center}
\begin{tikzcd}
    \text{rigid integrable $G$-connections on $X$} \arrow[rr, "(1)"] \arrow[dd, "(3)"] & & \text{rigid $G$-Higgs bundles on $X$} \arrow[d,"(2)"]\\
    & & \text{rigid nilpotent $G$-Higgs bundles on $X$} \arrow[d, "(3)"] \\
    \text{rigid integrable $G$-connections on $X_s$} & & \text{rigid nilpotent $G$-Higgs bundles on $X_s$} \arrow[ll, "(4)"]
\end{tikzcd}
\end{center}
\begin{enumerate}
    \item[(1)] By Corlette--Simpson correspondence (Theorem \ref{thm_CS_G}), rigid integrable $G$-connections correspond to rigid $G$-Higgs bundles.
    \item[(2)] Rigid $G$-Higgs bundles are actually nilpotent (Lemma \ref{lem:rigid-higgs-nilpotent}).
    \item[(3)] The approach to find such an arithmetic model is the same as \cite[Lemma 3.1 and Proposition 3.3]{EG20}.
    \item[(4)] The nonabelian Hodge correspondence in positive characteristic for principal bundles (Theorem \ref{thm_main_SSW24}) preserves rigidity (Lemma \ref{lem_frobenius_pullback_rigid}), and thus the corresponding rigid integrable $G$-connections on $X_s$ have nilpotent $p$-curvatures.
\end{enumerate}
Furthermore, Esnault and Groechenig constructed a Frobenius pullback map for the stack of de Rham local systems over the Witt vector ring in their recent work \cite{EG25}. This new strategy provides a new proof of the nilpotency of the $p$-curvatures of rigid flat connections \cite[Corollary 1.3]{EG25}. Although we will not discuss this new approach in this paper, it seems that the main result (\cref{thm:main}) can be proved in a similar way by discussing the Frobenius pullback map for the stack of de Rham $G$-local systems over the Witt vector ring.

\subsection*{Acknowledgement} The authors thank Raju Krishnamoorthy for his lecture notes from a seminar on Esnault--Groechenig's work at Humboldt-Universität zu Berlin and for valuable discussions. The authors also express their gratitude to Mao Sheng and Jianping Wang for numerous insightful discussions on exponential maps and the nonabelian Hodge correspondence in positive characteristic. The author would like to thank the anonymous referee to point out a gap about the horizontal property of Hitchin morphisms on higher dimensional varieties.

Huang acknowledges fundings from the European Research Council (ERC) under the European Union's Horizon 2020 research and innovation program (grant agreement No. 101018839) and Deutsche Forschungsgemeinschaft (DFG, Projektnummer 547382045). Qin is supported by the European Research Council (ERC) under the European Union's Horizon 2020 research and innovation program (grant agreement no. 101020009, project TameHodge). Sun is
partially supported by Guangdong Basic and Applied Basic
Research Foundation (No. 2024A1515011583).

\section{Recollections on nonabelian Hodge correspondence}

In this section, we briefly review the nonabelian Hodge correspondence in mixed characteristics. Although we mainly focus on the case of curves in this paper, some results hold for smooth projective varieties in arbitrary dimension.

\subsection{\texorpdfstring{$G$}{G}-Higgs bundles and flat \texorpdfstring{$G$}{G}-bundles}\label{subsect_Higgs_conn}

Since we consider both $G$-Higgs bundles and flat $G$-bundles in mixed characteristic, the following definitions of $G$-Higgs bundles and flat $G$-bundles are given over a perfect field $k$. Let $X$ be a smooth projective algebraic curve over $k$. In the case of positive characteristic, we have the following diagram
\begin{center}
\begin{tikzcd}
X \arrow[rr, "F" description] \arrow[rrd] \arrow[rrrr, bend left, "F_X" description] & & X' \arrow[d] \arrow[rr, "\pi_{X/k}" description] & & X \arrow[d]\\
& & {\rm Spec} \, k \arrow[rr, "F_k" description] & & {\rm Spec} \, k
\end{tikzcd}
\end{center}
where
\begin{itemize}
    \item $F_X$ and $F_k$ are the absolute Frobenius morphisms;
    \item $X'$ is the pullback of $X$ along $F_k$;
    \item $F: X \rightarrow X'$ is the relative Frobenius morphism.
\end{itemize}
Let $\Omega_X$ denote the cotangent sheaf and $G$ be a connected split reductive group $G$ over $k$ with Lie algebra $\mathfrak{g}$. Given a $G$-bundle $V$ on $X$, we denote by $V(\mathfrak{g}):= V \times_G \mathfrak{g}$ the adjoint bundle.

\subsubsection{\texorpdfstring{$G$}{G}-Higgs bundles}

\begin{definition}
A \emph{$G$-Higgs bundle} on $X$ is a pair $(V,\theta)$, where $V$ is a $G$-bundle and $\theta \in H^0(X, V(\mathfrak{g}) \otimes \Omega_X)$ is a section satisfying the integrability condition $\theta \wedge \theta =0$. Such a section $\theta$ is called a \emph{Higgs field}.
\end{definition}

Let $(V,\theta)$ be a $G$-Higgs bundle. Taking a parabolic subgroup $P \subseteq G$ and a reduction of structure group $\sigma: X \rightarrow V/P$, we define the product $V_\sigma$ via the following diagram
\begin{center}
\begin{tikzcd}
V_\sigma \arrow[rr, dotted] \arrow[d, dotted] & & V \arrow[d] \\
X \arrow[rr, "\sigma"] & & V/P \, .
\end{tikzcd}
\end{center}
Clearly, $V_\sigma$ is a $P$-bundle on $X$. A reduction of structure group $\sigma$ is \emph{compatible} with the Higgs field $\theta$ if there is a lifting
\begin{center}
\begin{tikzcd}
& & V_\sigma(\mathfrak{p}) \otimes \Omega_X \arrow[d] \\
X \arrow[rr, "\theta"'] \arrow[urr, "\theta_\sigma", dotted] & & V(\mathfrak{g}) \otimes \Omega_X \, ,
\end{tikzcd}
\end{center}
where $V_\sigma(\mathfrak{p})$ is the adjoint bundle of $V_\sigma$. Taking a character $\chi: P \rightarrow \mathbb{G}_m$, we obtain a $\mathbb{G}_m$-bundle $\chi_*(V_\sigma)$ and thus a line bundle on $X$. Now, we will give the $R$-stability condition on $G$-Higgs bundles, which is equivalent to the slope stability condition. We refer the reader to \cite{Ram75,Ram96a,Ram96b} for Ramanathan's original definition for principal bundles and \cite{KSZ24} for parahoric objects.

\begin{definition}\label{defn_stab_Higgs}
A $G$-Higgs bundle $(V,\theta)$ is \emph{$R$-semistable} (resp. \emph{$R$-stable}) if
\begin{itemize}
    \item for any proper parabolic subgroup $P \subseteq G$;
    \item for any $\theta$-compatible reduction of structure group $X \rightarrow V/P$;
    \item for any anti-dominant character $\chi: P \rightarrow \mathbb{G}_m$ acting trivially on the center of $P$,
\end{itemize}
we have
\begin{align*}
    \deg \chi_*(V_\sigma) \geq 0 \, (\text{resp.} \, >) \, .
\end{align*}
\end{definition}

\subsubsection{Flat $G$-bundles}

Now we move to flat $G$-bundles and refer the reader to \cite[Appendix]{CZ15} for more details.

\begin{definition}
Let $E$ be a $G$-bundle. An \emph{integrable $G$-connection $\nabla$} on $E$ is an integrable connection $\nabla: \mathcal{O}_E \rightarrow \mathcal{O}_E \otimes \Omega_X$ compatible with the $G$-action on $E$. Equivalently, an \emph{integrable $G$-connection} is a section $T_X \rightarrow {\rm At}(E)$ of Lie algebroids of the following short exact sequence
\begin{align*}
    0 \rightarrow E(\mathfrak{g}) \rightarrow {\rm At}(E) \rightarrow T_X \rightarrow 0,
\end{align*}
where ${\rm At}(E)$ is the Atiyah Lie algebroid of $E$. A \emph{flat $G$-bundle} is a pair $(E,\nabla)$, where $E$ is a $G$-bundle and $\nabla$ is an integrable $G$-connection. Sometimes, such a pair $(E,\nabla)$ is called an \emph{integrable $G$-connection} for convenience. Denote by
\begin{align*}
    \psi:=\psi_{\nabla}: F^* T_{X'} \rightarrow E(\mathfrak{g})
\end{align*}
the \textit{$p$-curvature} associated to $\nabla$, which is adjoint to the map
\begin{align*}
    T_X \rightarrow E(\mathfrak{g}), \quad v \mapsto \nabla(v)^p - \nabla(v^{[p]}).
\end{align*}
\end{definition}

Let $(E,\nabla)$ be an integrable $G$-connection. Given a reduction of structure group $\varsigma: X \rightarrow E/P$, we have a map $E_\varsigma \rightarrow E$. The reduction of structure group $\varsigma$ is \emph{$\nabla$-compatible} if the connection $\nabla$ induces an integrable connection $\nabla_\varsigma: \mathcal{O}_{V_\varsigma} \rightarrow \mathcal{O}_{V_\varsigma} \otimes \Omega_{X}$ such that the diagram commutes
\begin{center}
	\begin{tikzcd}
		\mathcal{O}_V \arrow[r,"\nabla"] \arrow[d] & \mathcal{O}_V \otimes \Omega_{X} \arrow[d]  \\
		\mathcal{O}_{V_\varsigma} \arrow[r,"\nabla_\varsigma", dotted] & \mathcal{O}_{V_\varsigma} \otimes \Omega_{X} \, .
	\end{tikzcd}
\end{center}

\begin{definition}\label{defn_stab_conn}
An integrable $G$-connection $(E,\nabla)$ is \emph{$R$-semistable} (resp. \emph{$R$-stable}), if
\begin{itemize}
\item for any proper parabolic subgroup $P \subseteq G$;
\item for any $\nabla$-compatible reduction of structure group $\varsigma: X \rightarrow E/P$;
\item for any antidominant character $\chi: P \rightarrow \mathbb{G}_m$ acting trivially on the center of $P$,
\end{itemize}
we have
\begin{align*}
	\deg \chi_* (V_{\varsigma}) \geq 0 \text{ (resp. $> 0$)}.
\end{align*}
\end{definition}

\subsubsection{Nilpotency}\label{subsect_nil}

We start with the case of characteristic zero. Let $k$ be a field of characteristic zero and let $R$ be a finitely generated $k$-algebra. We introduce the following definition of nilpotency.

\begin{definition}\label{defn_nil0}
    An element $x \in \mathfrak{g}(R)$ is \emph{nilpotent of exponent $\leq n-1$} if it lies in a nilpotent Lie subalgebra of nilpotency class $\leq n-1$. A collection of elements $\{x_i\}_{i \in I} \subseteq \mathfrak{g}(R)$, where $I$ is an index set, is \emph{nilpotent of exponent $\leq n-1$} if they are included in a nilpotent Lie subalgebra of nilpotency class $\leq n-1$.
\end{definition}

If $x \in \mathfrak{g}(R)$ is nilpotent, it corresponds to an element in $G(R)$ via the exponential map. We briefly review this well-known fact and refer the reader to \cite[Chapitre II. and Chapitre IV.]{DG70} for more details and relevant properties of the exponential map. There is an exact sequence
\begin{align*}
    0 \rightarrow \mathfrak{g}(R) \rightarrow G( R[\varepsilon] / (\varepsilon^2) ) \rightarrow G(R) \rightarrow 0.
\end{align*}
For each element $x \in \mathfrak{g}(R)$, there exists a unique element in $G(R[[T]])$ induced by the morphism $\mathfrak{g}(R) \rightarrow G( R[\varepsilon] / (\varepsilon^2) )$ and denote it by ${\rm exp}(Tx)$. If $x$ is nilpotent, we have ${\rm exp}(Tx) \in G(R[T])$, and thus we obtain an element in $G(R)$ by taking $T=1$ and denote it by ${\rm exp}(x)$. Moreover, the above discussion induces an isomorphism $\mathfrak{u} \rightarrow U$ via Cambell--Hausdorff series, where $U$ is a unipotent algebraic group and $\mathfrak{u}$ is its Lie algebra. This morphism is denoted by ${\rm exp} : \mathfrak{u} \rightarrow U$ and is called the exponential map. Now we introduce the definition of nilpotent $G$-Higgs bundles \cite[Definition 2.10]{SSW24b}.

\begin{definition}[Definition 2.10 in \cite{SSW24b}]\label{defn_HIG_nil_0}
    A $G$-Higgs bundles $(V,\theta)$ on $X$ is \emph{nilpotent of exponent $\leq n-1$} if there exists a covering of $X$ by open affine subsets $U$ such that the set $\{ \theta|_{U} (\partial) \, | \, \text{$\partial$ is a section of $T_{U}$} \} \subseteq \mathfrak{g}(\mathcal{O}_{U})$ is nilpotent of exponent $\leq n-1$.
\end{definition}

For positive characteristics, the nilpotent elements considered in this paper are always assumed to be obtained by a mod $p$ reduction. Here is a precise description. Let $k$ be a perfect field in characteristic $p$. There is a natural morphism $\mathbb{Z}_{(p)} \rightarrow \mathbb{F}_p \rightarrow k$. Following the argument in \cite[\S 4]{Ser94} and \cite[\S 5]{Sei00}, there is an isomorphism $\overline{\rm exp}: \mathfrak{u} \rightarrow U$, where $U$ is a unipotent algebraic group over $k$ of nilpotency class $\leq p-1$ and $\mathfrak{u}$ is the Lie algebra of $U$, induced by the exponential map in characteristic zero via the base change $\mathbb{Z}_{(p)} \rightarrow k$. An element $x \in \mathfrak{g}$ is \emph{nilpotent of exponent $\leq p-1$}, if it is included in a nilpotent Lie subalgebra of nilpotency class $\leq p-1$ obtained in this way. Now we give the definition of nilpotent $G$-Higgs bundles and nilpotent integrable $G$-connections in positive characteristics.

\begin{definition}\label{defn_nil}
A $G$-Higgs bundles $(V,\theta)$ (resp. integrable $G$-connections $(E,\nabla)$) on $X$ is \emph{nilpotent of exponent $\leq p-1$} if there exists a covering of $X$ by open affine subsets $U$ such that the set $\{ \theta|_{U} (\partial) \, | \, \text{$\partial$ is a section of $T_{U}$} \}$ (resp. $\{ \psi |_{U} (\partial) \, | \, \text{$\partial$ is a section of $F^* T_{U'}$} \}$) is nilpotent of exponent $\leq p-1$. Moreover, such $G$-Higgs bundles (resp. integrable $G$-connections) are also called $[p]$-nilpotent.
\end{definition}

\begin{remark}
We would like to make a comment on the definition of nilpotency we give here. When $G={\rm GL}_n$, Definition \ref{defn_nil0} is equivalent to that given by Katz \cite{Kat70} (see a proof in \cite[Lemma 2.9]{SSW24b}). It is also natural to compare the definition with the one that $x^{[p]} = 0$ for $x \in \mathfrak{g}$ in positive characteristic, where $x^{[p]} \in \mathfrak{g}$ is the $p$-th power of $x$. When $G={\rm GL}_n$ and $p<n$, the condition $x^{[p]} = 0$ is not equivalent to the condition introduced by Katz. On the other hand, if $p$ is large enough, these two definitions are equivalent.
\end{remark}

\subsection{Nonabelian Hodge correspondence in characteristic zero}\label{subsect_CS_corr}

In this subsection, we briefly overview the classical nonabelian Hodge correspondence, which is also called the Corlette---Simpson correspondence \cite{Sim94a,Sim94b}. Although $X$ is a smooth projective curve, the correspondence holds for any complex smooth projective variety. The nonabelian Hodge correspondence is given as a one-to-one correspondence among three objects:
\begin{itemize}
    \item (poly)stable Higgs bundles on $X$ with vanishing Chern classes;
    \item (semi)simple flat bundles on $X$;
    \item (semi)simple ${\rm GL}$-representations of the fundamental group $\pi_1(X)$.
\end{itemize}
Simple flat bundles (or integrable connections) are also called irreducible flat bundles (or integrable connections). Moreover, this correspondence induces homeomorphisms among three moduli spaces
\begin{align*}
    \mathcal{M}_{\rm Dol}(X,n) \cong \mathcal{M}_{\rm dR}(X,n) \cong \mathcal{M}_{\rm B}(X,n),
\end{align*}
where
\begin{itemize}
    \item $\mathcal{M}_{\rm Dol}(X,n)$ is the moduli space of semistable Higgs bundles of rank $n$ on $X$ with vanishing Chern classes, which is called the \emph{Dolbeault moduli space};

    \item $\mathcal{M}_{\rm dR}(X,n)$ is the moduli space of semisimple flat bundles of rank $n$ on $X$, which is called the \emph{de Rham moduli space};

    \item $\mathcal{M}_{\rm B}(X,n)$ is the moduli space of semisimple ${\rm GL}_n$-representations of the fundamental group, which is called the \emph{Betti moduli space}.
\end{itemize}

Now, we come to principal objects. Let $G$ be a connected complex reductive group. We recall the following definitions from \cite[\S 6, Principal objects]{Sim92}. A $G$-Higgs bundle is \emph{of vanishing Chern classes} if the adjoint Higgs bundle is of vanishing Chern classes. Moreover, it implies that the associated Higgs bundle is of vanishing Chern classes for any faithful representation. A $G$-Higgs bundle is \emph{(semi)stable} if there exists a faithful representation of $G$ such that the associated Higgs bundle is $P$-(semi)stable. Then, the Corlette--Simpson correspondence for principal objects is given as follows:
\begin{theorem}[{\cite[Theorem 9.11 \& Lemma 9.14]{Sim94b}}]\label{thm_CS_G}
    There is a homeomorphism between $\mathcal{M}_{\rm Dol}(X,G)$ and $\mathcal{M}_{\rm dR}(X,G)$ and there is a complex analytic isomorphism between $\mathcal{M}_{\rm dR}(X,G)$ and $\mathcal{M}_{\rm B}(X,G)$, where
    \begin{itemize}
    \item $\mathcal{M}_{\rm Dol}(X,G)$ is the moduli space of semistable $G$-Higgs bundles on $X$ with vanishing Chern classes;

    \item $\mathcal{M}_{\rm dR}(X,G)$ is the moduli space of semisimple integrable $G$-connections on $X$;

    \item $\mathcal{M}_{\rm B}(X,G)$ is the moduli space of semisimple $G$-representations of the fundamental group.
\end{itemize}
\end{theorem}

\begin{remark}\label{rem_nahc_G}
In the case of curves, the $R$-stability condition of $G$-bundles is equivalent to the slope stability condition for the associated vector bundles under a faithful representation $G \rightarrow {\rm GL}(W)$ (see \cite[Lemma 3.3]{Ram75} and \cite[Proposition 3.17]{Ram96a} for instance).
\end{remark}

%\begin{remark}
%Simpson gave the nonabelian Hodge correspondence in a relative version for a smooth projective morphism $\mathcal{X} \rightarrow S$ \cite{Sim94b}, where $S$ is a base scheme of finite type over $\mathbb{C}$. In this paper, we need to consider a special case that $X_S := X \times_{{\rm Spec} \, \mathbb{C}} S$, where $S$ is a scheme of finite type over $\mathbb{C}$, and we denote $\mathcal{M}_{\rm Dol}(X_S/S,G)$ and $\mathcal{M}_{\rm dR}(X_S/S,G)$ the corresponding Dolbeault and de Rham moduli spaces respectively.
%\end{remark}

\subsection{Nonabelian Hodge correspondence in positive characteristics}\label{subsect_nahc_p}

We review the nonabelian Hodge correspondence for principal objects in positive characteristic, which is first studied by Ogus--Vologodsky in the nilpotent case \cite{OV07} and then fully established by Chen--Zhu on curves \cite{CZ15}.

Denote by $\mathfrak{g}$ the Lie algebra of $G$, and there is a natural $\mathbb{G}_m$-action on $\mathfrak{g}$. Consider the stack $[\mathfrak{g} / (G \times \mathbb{G}_m)]$. Let $X \rightarrow [\ast / \mathbb{G}_m]$ be the morphism corresponding to the cotangent bundle $T^*_X$. Note that there is a natural morphism $[\mathfrak{g} /  (G \times \mathbb{G}_m)] \rightarrow [\ast / \mathbb{G}_m]$. Then the moduli stack ${\rm HIG}(X,G)$ of $G$-Higgs bundles on $X$ is defined as sections of the fiber product $X \times_{[\ast / \mathbb{G}_m]} [\mathfrak{g} /  (G \times \mathbb{G}_m)]$, i.e.
\begin{align*}
    {\rm HIG}(X,G) := {\rm Sect}(X,X \times_{[\ast / \mathbb{G}_m]} [\mathfrak{g} /  (G \times \mathbb{G}_m)]).
\end{align*}
The natural map $[\mathfrak{g} / G ] \rightarrow \mathfrak{g} / \! \! / G$ induces a morphism of stacks
\begin{align*}
    h_{X,G}: {\rm HIG}(X,G) \rightarrow \mathscr{A}(X,G),
\end{align*}
where
\begin{align*}
    \mathscr{A}(X,G) := {\rm Sect}(X, X \times_{ [\ast / \mathbb{G}_m]} [ (\mathfrak{g} / \! \! / G) / \mathbb{G}_m  ] ).
\end{align*}
This morphism $h_{X,G}$ is known as the \emph{Hitchin morphism} and if there is no ambiguity, we use the notation $h$ for convenience. Via the Chevalley restriction theorem $\mathfrak{g}  / \! \! / G \cong \mathfrak{t}/W$, $\mathscr{A}(X,G)$ is exactly the Hitchin base. Indeed, the above construction works in mixed characteristics, and thus the Hitchin morphism induces one on the Dolbeault moduli space $\mathcal{M}_{\rm Dol}(X,G)$ and we use the same notation for simplicity.

Let ${\rm MIC}(X,G)$ be the stack of integrable $G$-connections on $X$ (or stack of $G$-local systems on $X$). There is a morphism
\begin{align*}
    h_p: {\rm MIC}(X,G) \rightarrow \mathscr{A}(X',G)
\end{align*}
induced by the $p$-curvature, which is called the $p$-Hitchin morphism \cite[Proposition 3.1]{CZ15}.

Fixing a $W_2(k)$-lifting of $X$ and restricting to the nilpotent case, Chen--Zhu's result \cite[Theorem 1.2]{CZ15} recovers Ogus--Vologodsky correspondence \cite{OV07} and gives the following diagram
\begin{center}
\begin{tikzcd}
    {\rm HIG}_{p-1}(X',G) \arrow[rd,"h" description] \arrow[rr, out=10, in=170, dotted, "C^{-1}" description] & & {\rm MIC}_{p-1}(X,G) \arrow[ld, "h_p" description] \arrow[ll, out=190, in=350, dotted, "C" description] \\
    & \mathscr{A}(X',G) & \quad ,
\end{tikzcd}
\end{center}
where $C$ is the Cartier transform and $C^{-1}$ is the inverse Cartier transform. Moreover, de Cataldo--Groechenig--Zhang's work shows that the correspondence preserves the $R$-stability condition \cite{deGZ24}.

\section{Nilpotency of \texorpdfstring{$p$}{p}-curvatures}\label{sect_main}

\subsection{Rigidity and existence of nice models}

In this subsection, we go back to the setup in \cref{subsect_CS_corr} about the Corlette--Simpson correspondence and let $X$ be a complex smooth projective curve.

\begin{definition}
  For a scheme $\cc{X}\to S$, we denote by $\cc{X}^{\rm rig}$ the maximal open subscheme such that $\cc{X}^{\rm rig}\to S$ is quasi-finite at all points of $\cc{X}^{\rm rig}$.
\end{definition}

\begin{definition}\label{defn_rigid}
    A $R$-stable $G$-Higgs bundle (resp. $R$-stable integrable $G$-connection) is called \emph{rigid} if every nearby $G$-Higgs bundle (resp. integrable $G$-connection) can be deformed to it, or equivalently, if the corresponding point in the Dolbeault moduli space $\mathcal{M}_{\text{Dol}}(X_S/S,G)$ (resp. de Rham moduli space $\mathcal{M}_{\text{dR}}(X_S/S,G)$) is isolated. Denote by $\mathcal{M}^{\rm rig}_{\rm Dol}(X_S/S,G) \subseteq \mathcal{M}_{\rm Dol}(X_S/S,G)$ (resp. $\mathcal{M}^{\rm rig}_{\rm dR}(X_S/S,G) \subseteq \mathcal{M}_{\rm dR}(X_S/S,G)$) the open subscheme of rigid $G$-Higgs bundles (resp. integrable $G$-connections).
\end{definition}

Based on the above definition, here is an equivalent description of rigidity from deformation theory \cite[proof of Lemma 3.4]{EG20}, which does not depend on the existence of the moduli space and work in mixed characteristics. We only give the statement for $G$-Higgs bundles.

\begin{lemma}\label{lem_rig}
A $G$-Higgs bundle $(V,\theta)$ is not rigid if and only if there exists a positive dimension geometrically irreducible $k$-scheme $C$ of finite type, a $C$-family of $G$-Higgs bundles $(V_C,\theta_C)$ and two closed points $c_0,c_1 \in C(k')$ defined over a finite field extension $k'/k$ such that $(V_C,\theta_C)_{c_0} \cong (V,\theta)_{k'}$ and $(V_C,\theta_C)_{c_0}$ and $(V_C,\theta_C)_{c_1}$ are not isomorphic over the algebraic closure $\bar{k}$.
\end{lemma}

The following definition is a lemma given in \cite[Lemma 3.1]{EG20}.

\begin{definition}
A scheme $S$ is called an \emph{arithmetic model} if it satisfies the following conditions
\begin{enumerate}
    \item[(1)] $S$ is of finite type and smooth over ${\rm Spec} \, \mathbb{Z}$;
    \item[(2)] $S$ has a unique generic point $\eta$ and there is an embedding of fields $k(\eta) \subseteq \mathbb{C}$;
    \item[(3)] we have ${\rm Spec} \, \mathbb{C} \times_S X_S \cong X$, where the product is taken along the embedding given in $(2)$.
\end{enumerate}
\end{definition}

\begin{proposition}[{\cite[Proposition 3.3]{EG20}}]\label{prop:nice-models}
  There exists an arithmetic model $S$ such that
  \begin{itemize}
    \item Each rigid integrable $G$-connection $(E,\nabla)$ can be spread out to $(E_S,\nabla_S)$ on $X_S$, which is $R$-stable over geometric points.

    \item For every rigid $G$-Higgs bundle $(V,\theta)$, there exists a spreading-out $(V_S,\theta_S)$ which is $R$-stable over geometric points and $\theta_S$ is nilpotent.

    \item  $(E_S,\nabla_S)$ (resp. $(V_S,\theta_S)$) induces a section
    \begin{align*}
        [E_S,\nabla_S] : S\to\cc{M}_{\rm dR}(X_S/S, G) \quad (resp. \, \, [V_S,\theta_S]: S\to\cc{M}_{\rm Dol}(X_S/S, G))
    \end{align*}
    factoring through $\cc{M}_{\rm dR}^{\rm rig}(X_S/S, G)$ (resp. $\cc{M}_{\rm Dol}^{\rm rig}(X_S/S, G)$).

    \item For each $y \in |\cc{M}_{\rm dR}^{\rm rig}(X_S/S, G)|$ (resp. $y \in |\cc{M}_{\rm Dol}^{\rm rig}(X_S/S, G)|$), there exists $(E_S,\nabla_S)$ (resp. $(V_S,\theta_S)$) such that $y$ belongs to the set-theoretic image $[E_S,\nabla_S](|S|)$ (resp. $[V_S,\theta_S](|S|)$).
  \end{itemize}
\end{proposition}

\begin{remark}
In \cite[\S 3]{EG20}, the authors fix a line bundle $L$ throughout the paper, treated as the fixed determinant line bundle. Therefore, the case they consider is actually for rigid (twisted) integrable ${\rm SL}_n$-connections.
\end{remark}

\subsection{Proof of the main result}

In this section, $X$ is always a connected smooth projective curve over $\mathbb{C}$, and we use the notation $Z$ for a smooth projective curve over a perfect field $k$ in mixed characteristics.

\begin{lemma}\label{lem:rigid-higgs-nilpotent}
Rigid $G$-Higgs bundles over $Z$ have nilpotent Higgs field.
\end{lemma}

\begin{proof}
As vector spaces, scalar multiplication gives a $\mathbb{G}_m$-action on $\mathfrak{g}$, and thus induces a $\mathbb{G}_m$-action on $\mathfrak{g} / \! \! / G \cong \mathfrak{t}/W$. Then the above $\mathbb{G}_{m}$-actions induce ones on the moduli stacks ${\rm HIG}(X,G)$ and $\mathscr{A}(X,G)$ respectively, which are compatible with the Hitchin morphism $h: {\rm HIG}(X,G) \rightarrow \mathscr{A}(X,G)$.

With respect to the $\mathbb{G}_m$-actions given above, the proof of this lemma is an analog of \cite[Lemma 2.1]{EG20}. We include it here for completeness. Let $(V,\theta)$ be a rigid $G$-Higgs bundle on $Z$. If $\theta$ is not nilpotent, then image $h( (V,\theta) )$ will be nontrivial. Since $\mathbb{G}_m$-action is compatible with the Hitchin morphism, we obtain a natural $\mathbb{G}_m$-family $(V,\lambda \theta)$, which is a nontrivial deformation of $(V,\theta)$. By Lemma \ref{lem_rig}, this is a contradiction.
\end{proof}

\begin{lemma}\label{lem:cartier-nilpotent}
Suppose that the characteristic $p$ is pretty good and big. Then the inverse Cartier transform $C^{-1}$ given in \cref{subsect_nahc_p} sends a rigid $G$-Higgs bundle on $Z'$ to an integrable $G$-connection on $Z$ with nilpotent $p$-curvature.
\end{lemma}

\begin{proof}
    By Lemma \ref{lem:rigid-higgs-nilpotent}, a rigid $G$-Higgs bundle $(V,\theta)$ on $Z'$ is nilpotent. Therefore, the integrable $G$-connection $C^{-1}(V,\theta)$ has nilpotent $p$-curvature due to the commutativity of the following diagram given in \cref{subsect_nahc_p}.
    \begin{center}
    \begin{tikzcd}
    {\rm HIG}_{p-1}(Z',G) \arrow[rd,"h" description] \arrow[rr, "C^{-1}" description] & & {\rm MIC}_{p-1}(Z,G) \arrow[ld, "h_p" description]  \\
    & \mathscr{A}(Z',G) &
    \end{tikzcd}
    \end{center}
\end{proof}

\begin{lemma}\label{lem_frobenius_pullback_rigid}
With the same setup as in Lemma \ref{lem:cartier-nilpotent}, a $R$-stable $G$-Higgs bundle $(V,\theta)$ (resp. $R$-stable integrable $G$-connection $(E,\nabla)$) on $Z$ is rigid if and only if the $R$-stable $G$-Higgs bundle $\pi^*_{Z/k}(V,\theta)$ (resp. the $R$-stable integrable $G$-connection $\pi^*_{Z/k}(E,\nabla)$) on $Z'$ is rigid.
\end{lemma}

\begin{proof}
    This is an analog of \cite[Lemma 3.4]{EG20} with respect to the equivalent definition of rigidity in Lemma \ref{lem_rig}.
\end{proof}

\begin{proposition}\label{prop_OV_corr_rigid}
  Let $S$ be an arithmetic model for a given complex smooth projective variety $X$ given in \cref{prop:nice-models}. There exists a positive integer $N$, depending on the given data $(X,S,G)$, such that for any closed point $s \in S$ with $\rr{char} \, k(s) > N$, and any rigid $G$-Higgs bundle $(V_s,\theta_s)$, the inverse Cartier transform $C^{-1}(V'_s,\theta'_s)$ is a rigid integrable $G$-connection, where $(V'_s,\theta'_s):=\pi^*_{X/k}(V_s,\theta_s)$.
\end{proposition}

\begin{proof}
The rigidity is defined on stable objects and the equivalence of stability conditions follows from Lemma \ref{lem_frobenius_pullback_rigid} and Theorem \ref{thm_main_SSW24}. We only have to check the rigidity.

Let $D$ be an integer bigger than the degree of finite morphisms $\mathcal{M}_{\ast}^{\rm rig}(X_S/S)\to S$, where $\ast= {\rm Dol} \text{ or } {\rm dR}$. Given $s \in S$, let $(V_s,\theta_s)$ be a rigid $G$-Higgs bundle on $X_s$. Then for any deformation $(V_B,\theta_B)$ over an Artinian base $B$, it corresponds to an element \begin{align*}
    \chi_{\rm Dol}: {\rm Spec} \, B \rightarrow \mathscr{A}'
\end{align*}
where $\mathscr{A}':=\mathscr{A}(X',G)$, and furthermore, it factors through $\mathscr{A}'^{(D)}$, the nilpotent thickening of $0$ of order $D$.

In the case of ${\rm SL}_n$, Esnault--Groechenig showed that there exists an integer $N'$ (depending on $D$ and $n$) such that when $ p> N'$, (twisted) ${\rm SL}_n$-Higgs bundles in the preimage of any scheme theoretic point $T \rightarrow \mathscr{A}^{(D)}(X',{\rm SL}_n)$ under the Hitchin morphism are $[p]$-nilpotent. This property is called OV-admissible in \cite{EG20}. For reductive groups, there is a similar result. We fix a faithful representation $G \hookrightarrow {\rm SL}_n$ and obtain a natural commutative diagram
\begin{center}
\begin{tikzcd}
    {\rm HIG}_{p-1}(X',G) \arrow[rr] \arrow[d,"h"] & & {\rm HIG}_{p-1}(X',{\rm SL}_n) \arrow[d,"h"]  \\
    \mathscr{A}(X',G) \arrow[rr] & & \mathscr{A}(X',{\rm SL}_n) \, ,
\end{tikzcd}
\end{center}
where $h: {\rm HIG}_{p-1}(X',{\rm SL}_n) \rightarrow \mathscr{A}(X', {\rm SL}_n)$ is the Hitchin morphism. Now we consider the nilpotent thickenings
\begin{align*}
    \mathscr{A}'^{(D)} =  \mathscr{A}^{(D)}(X',G) \rightarrow \mathscr{A}^{(D)}(X',{\rm SL}_n).
\end{align*}
Thus, any scheme theoretic point $T \rightarrow {\mathscr{A}'} ^{(D)}$ will be mapped to scheme theoretic point $T \rightarrow \mathscr{A}^{(D)}(X', {\rm SL}_n)$. Then any $G$-Higgs bundle in the preimage of $T \rightarrow {\mathscr{A}'} ^{(D)}$ under the Hitchin morphism $h$ will be sent to a ${\rm SL}_n$-Higgs bundle included in the preimage of the point $T \rightarrow \mathscr{A}^{(D)}(X', {\rm SL}_n)$. Now let $N:=N'$. When $p > N$, $G$-Higgs bundles in the preimage of any scheme theoretic point $T \rightarrow {\mathscr{A}'} ^{(D)}$ under the Hitchin morphism are $[p]$-nilpotent.

Now we choose $m>D$ and assume $p>N$, where $N$ is given as above. If $(E_s,\nabla_s):=C^{-1}(V'_s,\theta'_s)$ is not rigid, then there exists positive dimensional deformation $(E_T,\nabla_T)$ over $(X_s)_T$. Assume that either
\begin{enumerate}
    \item $\chi_{\rm dR}\colon T\to \mathscr{A}'$ factors through $\mathscr{A}'^{(m)}$; or
    \item $\chi_{\rm dR}\colon T\to \mathscr{A}'$ does not factor through $\mathscr{A}'^{(m)}$.
\end{enumerate}

In the first case, there exists a $T$-family of $G$-Higgs bundles $(V_T,\theta_T)$ such that
$$C^{-1}(V_T,\theta_T)=(E_T,\nabla_T).$$
Since $(V'_s,\theta'_s)$ is rigid by Lemma \ref{lem_frobenius_pullback_rigid}, $(V_T,\theta_T)$ is an infinitesimal deformation of order $\leq D$, which contradicts the condition that $\chi_{\rm dR}$ factors through $\mathscr{A}'^{(m)}$.

In the latter case, let $T^{(n)}$ be the $n$-th order neighborhood of $t\in T$ corresponding to $(E_s,\nabla_s)$. We have an induced family
$$(E_{T^{(n)}},\nabla_{T^{(n)}})$$
over $(X_s)_{ T^{(n)}}$, which also induces
$$\chi_{\rm dR}\colon T^{(n)}\to \mathscr{A}'.$$
There exists $n>m$ such that $\chi_{\rm dR}$ factors through $\mathscr{A}'^{(m)}$ but not through $\mathscr{A}'^{(m-1)}$. Moreover,

Since $p>N$, there exists $T^{(n)}$-family of rigid $G$-Higgs bundles such that
\begin{align*}
    \chi_{\rm dR} = \chi_{\rm Dol} :=h ((V_{T^{(n)}} , \theta_{ T^{(n)} })).
\end{align*}
Since $\chi_{\rm Dol}$ does not factor through $\mathscr{A}'^{(k)}$ for $k<m$, we obtain a strictly order $m$ deformation of $(V_s,\theta_s)$, contradicting $m>D$.
\end{proof}

\begin{definition}
Define $n_{\rm dR}(Z,G)$ (resp. $n_{\rm Dol}(Z,G)$) to be the number of isomorphism classes of rigid integrable $G$-connections (resp. $G$-Higgs bundles) on $Z$.
\end{definition}

\begin{proof}[Proof of Theorem \ref{thm:main}]
Let $S$ be a nice model given in Proposition \ref{prop:nice-models} and take $s \in S$ such that ${\rm char} \, k(s) > N$, where $N$ is the integer given in Proposition \ref{prop_OV_corr_rigid}. The properties of nice models imply
\begin{align*}
    n_{\rm dR}(X,G)=n_{\rm dR}(X_s,G), \quad n_{\rm Dol}(X,G)=n_{\rm Dol}(X_s,G).
\end{align*}
By Lemma \ref{lem_frobenius_pullback_rigid}, we know
\begin{align*}
    n_{\rm Dol}(X_s,G)=n_{\rm Dol}(X'_s,G).
\end{align*}
Furthermore, Corlette--Simpson correspondence (Theorem \ref{thm_CS_G}) gives
\begin{align*}
    n_{\rm dR}(X,G)=n_{\rm Dol}(X,G).
\end{align*}
Combining the above equalities, we have
\begin{align*}
    n_{\rm dR}(X_s,G) = n_{\rm Dol}(X'_s,G).
\end{align*}
The only thing left to show is that rigid integrable $G$-connections on $X_s$ has nilpotent $p$-curvature, i.e.,
\begin{align*}
    n_{\rm dR}^{\rm nilp}(X_s,G) = n_{\rm dR}(X_s,G),
\end{align*}
where $n_{\rm dR}^{\rm nilp}(X_s,G)$ is the number of stable rigid integrable $G$-connections on $X_s$ with nilpotent $p$-curvature. By Lemmas~\ref{lem:rigid-higgs-nilpotent} and~\ref{lem:cartier-nilpotent}, we have
\begin{align*}
    n_{\rm dR}^{\rm nilp}(X_s,G) = n_{\rm Dol}(X_s,G).
\end{align*}
Then, the desired equality comes from the following one
\begin{align*}
    n_{\rm dR}^{\rm nilp}(X_s,G) & =  n_{\rm Dol}(X_s,G)
     = n_{\rm dR}(X_s,G)
     \geq n^{\rm nilp}_{\rm dR}(X_s,G).
\end{align*}
This finishes the proof of this theorem.
\end{proof}

\section{Comments on the higher dimensional case}\label{sect_higher_dim}

In this section, we want to make some comments on applying our approach to the higher dimensional case.

\subsection{Stability conditions and moduli spaces}
Ramanathan introduces a natural stability condition on $G$-bundles \cite{Ram75,Ram96a}, which is called $R$-stability condition in this paper. On higher dimensional varieties, the $R$-stability condition of $G$-bundles is given by the reduction of structure group $U \rightarrow (V|_U) / P$ on an open subset $U$ such that ${\rm codim}(X \backslash U) \geq 2$ \cite[Definition 1.1]{AB01}. In such higher dimensional cases, a $G$-Higgs bundle is $R$-semistable if and only if its adjoint Higgs bundle is semistable \cite[Lemma 4.7]{AB01}. For a discussion in mixed characteristics, see \cite[\S 3.2]{GLSS08}.

In the context of the Corlette--Simpson correspondence, the Higgs bundles are always of vanishing Chern classes, which implies that all Higgs subsheaves are indeed subbundles \cite[Proposition 6.6]{Sim94b}. Furthermore, the Gieseker stability condition, also called the $P$-stability condition, is equivalent to the slope stability condition \cite[Remark below Corollary 6.7]{Sim94b}. Therefore, it is enough to consider (poly)stable Higgs bundles under the slope stability condition in the Corlette--Simpson correspondence.

Since the $R$-stability condition of $G$-bundles is equivalent to the slope stability condition of the associated vector bundles, the Dolbeault moduli space $\mathcal{M}_{\rm Dol}(X,G)$ of principal objects can be regarded as the moduli space of $R$-semistable $G$-Higgs bundles on $X$ with vanishing Chern classes.

\subsection{Nonabelian Hodge correspondence for $G$-bundles in positive characteristic}

The nonabelian Hodge correspondence for $G$-bundles on higher dimensional varieties in positive characteristic is still unknown. However, in \cite{SSW24b}, the authors give such a correspondence for nilpotent objects via the approach of exponential twisting introduced in \cite{LSZ15}. We only give the statement as follows and refer the reader to \cite[\S 2 \& \S 3]{SSW24b} for more details.

\begin{theorem}[{\cite[Theorem 3.1 \& Theorem 3.10]{SSW24b}}]\label{thm_main_SSW24}
Suppose that $X$ is $W_2(k)$-liftable and the characteristic $p$ is large enough. The category of nilpotent $G$-Higgs bundles on $X$ of exponent $\leq p-1$ and the category of nilpotent integrable $G$-connections on $X$ of exponent $\leq p-1$ are equivalent. Moreover, the equivalence preserves the $R$-stability condition.
\end{theorem}

After the establishing of the nonabelian Hodge correspondence, we aim at constructing a $p$-Hitchin morphism for integrable $G$-connections on higher dimensional varieties, which is still unknown. The main obstacle is that we do not know whether the Hitchin morphism is horizontal, although it is known in the case of curves \cite{BD}. If the horizontal property of the Hitchin morphism can be proved, the proof of the main theorem can be applied directly.

%\section*{Conflict of Interest Statement}
%On behalf of all authors, the corresponding author states that there is no conflict of interest.

\bibliography{bibtex}
\bibliographystyle{alpha}
\end{document}